\documentclass[12pt]{amsart}
\usepackage{SSdefn}

\usepackage{eucal}

\title[Proof of Stembridge's conjecture on Kronecker coefficients]{Proof of Stembridge's conjecture on \\stability of Kronecker coefficients}
\date{July 11, 2015}


\author{Steven V Sam}
\address{Department of Mathematics, University of California, Berkeley, CA}
\email{\href{mailto:svs@math.berkeley.edu}{svs@math.berkeley.edu}}
\urladdr{\url{http://math.berkeley.edu/~svs/}}

\author{Andrew Snowden}
\address{Department of Mathematics, University of Michigan, Ann Arbor, MI}
\email{\href{mailto:asnowden@umich.edu}{asnowden@umich.edu}}
\urladdr{\url{http://www-personal.umich.edu/~asnowden/}}

\thanks{SS was supported by a Miller research fellowship. AS was supported by NSF grant DMS-1303082}

\begin{document}

\begin{abstract}
We prove a conjecture of Stembridge concerning stability of Kronecker coefficients that vastly generalizes Murnaghan's theorem. The main idea is to identify the sequences of Kronecker coefficients in question with Hilbert functions of modules over finitely generated algebras. The proof only uses Schur--Weyl duality and the Borel--Weil theorem and does not rely on any existing work on Kronecker coefficients.
\end{abstract}

\maketitle

\section{Introduction}

\subsection{Stembridge's conjecture}

Given a partition $\lambda$ of $n$, let $\bM_{\lambda}$ denote the associated irreducible complex representation of the symmetric group $S_n$. The important {\bf Kronecker coefficients} $g_{\lambda,\mu,\nu}$ are the tensor product multiplicities:
\begin{displaymath}
\bM_{\mu} \otimes \bM_{\nu} \cong \bigoplus_{\lambda} \bM_{\lambda}^{\oplus g_{\lambda,\mu,\nu}}.
\end{displaymath}
One can attempt to understand these coefficients by studying their limiting behavior, in various senses. An important result in this direction is Murnaghan's observation (conjectured by Murnaghan in \cite{murnaghan} and proved by Littlewood in \cite[\S 4]{littlewood}): $g_{(d)+\lambda,(d)+\mu,(d)+\nu}$ is constant for $d \gg 0$. In \cite{stembridge}, Stembridge proposes a vast generalization of this result, centered on the following concept:

\begin{definition}[Stembridge]
A triple $(\alpha, \beta, \gamma)$ of partitions with $|\alpha| = |\beta| = |\gamma|$ and $g_{\alpha, \beta, \gamma} > 0$ is {\bf stable} if, for any other triple of partitions $(\lambda, \mu, \nu)$ with $|\lambda| = |\mu| = |\nu|$, the Kronecker coefficient $g_{d\alpha+\lambda, d\beta+\mu, d\gamma+\nu}$ is constant for $d \gg 0$.
\end{definition}

With this language, Murnaghan's theorem is the statement that the triple $((1), (1), (1))$ is stable. Stembridge's proposed generalization of Murnaghan's theorem is the following conjecture (\cite[Conj~4.3]{stembridge}):

\begin{conjecture} \label{mainconj}
A triple $(\alpha, \beta, \gamma)$ is stable if and only if $g_{d\alpha, d\beta, d\gamma}=1$ for $d > 0$.
\end{conjecture}

Murnaghan's theorem is an easy corollary of this conjecture since $g_{(d),(d),(d)}=1$ for all $d$: $\bM_{(d)}$ is the trivial representation of $S_d$, so $\bM_{(d)} \otimes \bM_{(d)} = \bM_{(d)}$.

Stembridge proves the ``only if'' direction of Conjecture~\ref{mainconj}. This paper proves the reverse direction, thus establishing the conjecture in full. In fact, we also reprove the ``only if'' direction, see \S\ref{sec:alg}. It may be worth emphasizing that our proof only uses Schur--Weyl duality and the Borel--Weil theorem and does not rely on any existing work on Kronecker coefficients.

\subsection{Statement of results}

For partitions $\lambda, \mu, \nu$ of $n$, let $\bK_{\lambda,\mu,\nu}$ be the $\bM_\lambda$-multiplicity space in $\bM_\mu \otimes \bM_\nu$, i.e.,
\begin{align} \label{eqn:Gspace}
\bM_\mu \otimes \bM_\nu = \bigoplus_\lambda \bK_{\lambda, \mu, \nu} \otimes \bM_\lambda.
\end{align}
This is a vector space whose dimension is the Kronecker coefficient $g_{\lambda, \mu, \nu}$. We will prove Conjecture~\ref{mainconj} by constructing some extra structure on these multiplicity spaces.

Let $\alpha$, $\beta$, $\gamma$, $\lambda$, $\mu$, $\nu$ be partitions such that $|\alpha| = |\beta| = |\gamma|$ and $|\lambda| = |\mu| = |\nu|$. Define
\begin{displaymath}
B_{\alpha,\beta,\gamma} = \bigoplus_{d \ge 0} \bK_{d\alpha,d\beta,d\gamma}, \qquad 
N_{\alpha,\beta,\gamma}^{\lambda,\mu,\nu} = \bigoplus_{d \ge 0} \bK_{d\alpha+\lambda,d\beta+\mu,d\gamma+\nu}.
\end{displaymath}
Our main result is:

\begin{theorem} \label{mainthm}
The space $B_{\alpha,\beta,\gamma}$ has the structure of a finitely generated graded integral domain, and $N^{\lambda,\mu,\nu}_{\alpha,\beta,\gamma}$ has the structure of a finitely generated torsion-free graded $B_{\alpha,\beta,\gamma}$-module.
\end{theorem}

\begin{corollary} \label{maincor}
Conjecture~\ref{mainconj} is true.
\end{corollary}

\begin{proof}
Suppose $g_{d\alpha, d\beta, d\gamma} = 1$ for all $d > 0$. Then $B_{\alpha, \beta, \gamma}$ is isomorphic to $\bC[t]$, where $t$ has degree one. It now follows from the structure theorem for finitely generated $\bC[t]$-modules that $N_{\alpha,\beta,\gamma}^{\lambda,\mu,\nu}$ is isomorphic to $\bigoplus_{i=1}^n B_{\alpha,\beta,\gamma}[r_i]$ for some $r_1, \ldots, r_n$, where $[r_i]$ denotes a shift in grading. We thus see that $g_{\lambda + d\alpha, \mu + d\beta, \nu + d\gamma}=n$ for $d \ge \max(r_1, \ldots, r_n)$.
\end{proof}

\begin{corollary}
Suppose $g_{\alpha,\beta,\gamma} \ne 0$, and let $r$ be the Krull dimension of $B_{\alpha, \beta, \gamma}$. Then $g_{\lambda + d\alpha, \mu + d\beta, \nu + d\gamma} \sim A d^{r-1}$ for a constant $A$ (depending on $\lambda,\mu,\nu$). More precisely, $g_{\lambda + d\alpha, \mu + d\beta, \nu + d\gamma}$ is a quasi-polynomial of degree $r-1$ with constant leading term for $d \gg 0$.
\end{corollary}

\begin{proof}
Let $N=N_{\alpha,\beta,\gamma}^{\lambda,\mu,\nu}$. The Hilbert series $\rH_N(t)$ has a pole of order $r$ at $t=1$, by standard properties of Krull dimension. Let $x$ be a non-zero degree one element of $B_{\alpha,\beta,\gamma}$. Since $N$ is torsion free, $\rH_N(t)=(1-t)^{-1} \rH_{N/xN}(t)$. As $N/xN$ has Krull dimension $r-1$, all poles of $\rH_{N/xN}(t)$ have order $\le r-1$. Thus $\rH_N(t)$ has a pole of order $r$ at $t=1$ and all other poles have order $\le r-1$ (and are at roots of unity). The corollary now follows from \cite[Theorem~4.1.1(iii)]{stanley}, and the fact that the $d$th coefficient of $\rH_N(t)$ is $g_{\lambda + d\alpha, \mu + d\beta, \nu + d\gamma}$.
\end{proof}

\begin{remark}
We will see in \S\ref{sec:alg} that the ring $B_{\alpha,\beta,\gamma}$ is normal and has rational singularities. We omitted this from the main result for simplicity and because it is not strictly needed for the application to Conjecture~\ref{mainconj}. However, normality can be used to prove the ``only if'' direction of Conjecture~\ref{mainconj}, as we will explain.
\end{remark}

\subsection{Related work} \label{ss:related-work}

A general result of Meinrenken and Sjamaar \cite[Corollary 2.12]{MS} implies that the function $d \mapsto g_{d\alpha, d\beta, d\gamma}$ is a quasi-polynomial for all $d \ge 0$. However, the same strengthening can fail for the general functions $d \mapsto g_{\lambda + d\alpha, \mu + d\beta, \nu + d\gamma}$.

Vallejo introduces a notion of {\it additive stability} in \cite{vallejo} and proves that it implies stability in Stembridge's sense. Additive stability is provided by the existence of a certain additive matrix, and hence is easier to apply, but it is less general \cite[Example 6.3]{vallejo}.

Pak and Panova show that for any $k \ge 1$, the triple $((1^k), (1^k), (k))$ is stable \cite[Theorem 1.1]{pak-panova}. This is also a special case of Vallejo's work just mentioned and Stembridge's result that $((k), \alpha, \alpha)$ is stable for any partition $\alpha$ of $k$ \cite[Example 6.3]{stembridge}.

Manivel uses geometric techniques in \cite{manivel1, manivel2} to produce many more examples of stable triples and to study the cone of stable triples.

Finally, see \cite[\S 4.5]{BV} for some general information about the degrees and leading coefficients of the Hilbert functions $d \mapsto g_{\lambda + d\alpha, \mu + d\beta, \nu + d\gamma}$.

\subsection{Outline of paper}

In \S\ref{sec:bg} we recall some background. The proof of Theorem~\ref{mainthm} is given in \S\ref{sec:proof}. Finally, \S\ref{sec:rmk} contains remarks and related results, including a Littlewood--Richardson version and a plethysm version of Theorem~\ref{mainthm}.

\subsection*{Acknowledgements}

We thank Greta Panova, Mateusz Micha{\l}ek, John Stembridge, and Ernesto Vallejo for helpful comments and references.

\section{Background} \label{sec:bg}

\subsection{Schur functors and Kronecker coefficients}

We refer to \cite[Part 1]{expos} for further explanation and references for this section.

For a partition $\lambda$ of $n$ define the Schur functor $\bS_{\lambda}$ by
\begin{displaymath}
\bS_{\lambda}(V)=\Hom_{S_n}(\bM_{\lambda}, V^{\otimes n})
\end{displaymath}
where $V$ is a complex vector space. We recall the well-known connection between Schur functors and Kronecker coefficients:

\begin{proposition} \label{prop:SW}
We have a natural identification
\begin{displaymath}
\bS_{\lambda}(V \otimes W) = \bigoplus_{\mu,\nu} \bK_{\lambda,\mu,\nu} \otimes \bS_{\mu}(V) \otimes \bS_{\nu}(W).
\end{displaymath}
\end{proposition}

\begin{proof}
We have decompositions
\begin{displaymath}
V^{\otimes n}=\bigoplus_{\mu} \bS_{\mu}(V) \otimes \bM_{\mu}, \qquad
W^{\otimes n}=\bigoplus_{\nu} \bS_{\nu}(W) \otimes \bM_{\nu}.
\end{displaymath}
Tensoring these together, and using the decomposition \eqref{eqn:Gspace}, we find
\[
(V \otimes W)^{\otimes n} = \bigoplus_{\lambda,\mu,\nu} \bK_{\lambda, \mu, \nu} \otimes \bS_{\mu}(V) \otimes \bS_{\nu}(W) \otimes \bM_{\lambda}.
\]
Taking the $\bM_{\lambda}$ isotypic component yields the stated result.
\end{proof}

Recall that $\bS_{\lambda}(V)$ is a nonzero irreducible representation of $\GL(V)$ if $\dim(V) \ge \ell(\lambda)$, and $0$ otherwise. Thus, assuming $\dim(V) \ge \ell(\lambda)$, we have a  natural isomorphism $\bC=(\bS_{\lambda}(V) \otimes \bS_{\lambda}(V^*))^{\GL(V)}$. Proposition~\ref{prop:SW} therefore gives a natural isomorphism
\begin{equation} \label{eq1}
\bK_{\lambda,\mu,\nu} = (\bS_{\lambda}(V^* \otimes W^*) \otimes \bS_{\mu}(V) \otimes \bS_{\nu}(W))^{\GL(V) \times \GL(W)},
\end{equation}
assuming $\dim(V) \ge \ell(\mu)$ and $\dim(W) \ge \ell(\nu)$.

We now recast \eqref{eq1} so that the right side reflects the symmetry of the left, at least superficially. Let $U$, $V$, and $W$ be finite dimensional vector spaces and let $\omega \colon U \times V \times W \to \bC$ be a trilinear form. Assume that $\omega$ is {\bf non-degenerate} in the sense that it induces an isomorphism $U \to V^* \otimes W^*$. Note that this implies that $\dim(U) = \dim(V) \dim(W)$. We let $G(\omega) \subset \GL(U) \times \GL(V) \times \GL(W)$ be the stabilizer of $\omega$; this projects isomorphically to $\GL(V) \times \GL(W)$. We can restate \eqref{eq1} as:

\begin{proposition}
\label{schur-kron}
Let $\omega \colon U \times V \times W \to \bC$ be a non-degenerate trilinear form and assume that $\dim(U) \ge \ell(\lambda)$, $\dim(V) \ge \ell(\mu)$, and $\dim(W) \ge \ell(\nu)$. Then we have a natural isomorphism
\begin{displaymath}
\bK_{\lambda,\mu,\nu} = (\bS_{\lambda}(U) \otimes \bS_{\mu}(V) \otimes \bS_{\nu}(W))^{G(\omega)}.
\end{displaymath}
\end{proposition}

\subsection{Invariant theory and Segre products}

\begin{theorem} \label{thm:GIT}
Let $G$ be a complex reductive group acting on a finitely generated $\bC$-algebra $A$ and also compatibly on a finitely generated $A$-module $M$, i.e., the multiplication map $A \otimes M \to M$ is $G$-invariant. 
\begin{enumerate}[\indent \rm (a)]
\item The ring of invariants $A^G$ is finitely generated.
\item $M^G$ is a finitely generated $A^G$-module.
\end{enumerate}
\end{theorem}

\begin{proof}
(a) See \cite[Theorem~3.6]{PV}. (b) See \cite[Theorem~3.25]{PV}.
\end{proof}

Let $V$ and $W$ be graded vector spaces. We define the {\bf Segre product} of $V$ and $W$ by
\begin{displaymath}
V \boxtimes W = \bigoplus_{d \ge 0} V_d \otimes W_d.
\end{displaymath}
This has the following interpretation in terms of invariant theory. The gradings on $V$ and $W$ are equivalent to algebraic $\bC^*$ actions. Thus $V \otimes W$ is naturally a representation of $(\bC^*)^2$, and $V \boxtimes W$ is the invariants under the diagonal subgroup $\{(a, a^{-1}) \mid a \in \bC^*\} \cong \bC^*$. From this, we get the following corollary.

\begin{corollary} \phantomsection \label{cor:segre} 
\begin{enumerate}[\indent \rm (a)]
\item If $A$ and $B$ are graded $\bC$-algebras then $A \boxtimes B$ is naturally a graded $\bC$-algebra. If $A$ and $B$ are finitely generated then so is $A \boxtimes B$. If in addition $A$ and $B$ are domains, then so is $A \boxtimes B$.

\item If $M$ and $N$ are graded $A$- and $B$-modules then $M \boxtimes N$ is naturally an $A \boxtimes B$ module. If $M$ is finitely generated over $A$ and $N$ is finitely generated over $B$ then $M \boxtimes N$ is finitely generated over $A \boxtimes B$. In addition, if $A$ and $B$ are domains and $M$ and $N$ are torsion-free modules, then $M \boxtimes N$ is a torsion-free $A \boxtimes B$ module.
\end{enumerate}
\end{corollary}

\section{Proof of Theorem~\ref{mainthm}} \label{sec:proof}

For partitions $\alpha$ and $\lambda$, define
\begin{displaymath}
A_{\alpha}=\bigoplus_{d \ge 0} \bS_{d\alpha}, \qquad
M_{\alpha,\lambda} = \bigoplus_{d \ge 0} \bS_{d\alpha+\lambda}.
\end{displaymath}

\begin{proposition} \label{prop:construct}
Let $U$ be a vector space with $\dim(U) \ge \ell(\alpha), \ell(\lambda)$. Then $A_{\alpha}(U)$ naturally has the structure of a finitely generated graded integral domain over $\bC$, and $M_{\alpha,\lambda}(U)$ naturally has the structure of a finitely generated torsion-free graded $A_{\alpha}(U)$-module.
\end{proposition}

\begin{proof}
Let $X$ be the flag variety of $\GL(U)$. For every partition $\alpha$ with $\ell(\alpha) \le \dim(U)$, there is a $G$-equivariant line bundle $\cL(\alpha)$ on $X$ whose sections are $\rH^0(X; \cL(\alpha)) = \bS_\alpha(U)$ (this is the Borel--Weil theorem, see \cite[\S 9.3]{fulton}) and they satisfy $\cL(\alpha) \otimes \cL(\beta) = \cL(\alpha + \beta)$. Let $\bV$ be the total space of the vector bundle $(\cL(\alpha) \oplus \cL(\lambda))^*$, and let $R=\rH^0(X; \Sym(\cL(\alpha) \oplus \cL(\lambda)))$ be the ring of global functions on $\bV$. This is an integral domain, since $\bV$ is an irreducible variety. It also has a bigrading given by $R_{n,m}=\rH^0(X; \cL(n \alpha + m\lambda))$. Since each $R_{n,m}$ is a (non-zero) irreducible representation of $\GL(V)$ (by the Borel--Weil theorem), and $R$ is an integral domain, it follows that $R$ is generated as a $\bC$-algebra by $R_{1,0}$ and $R_{0,1}$. In particular, $R$ is finitely generated as a $\bC$-algebra.

The bigrading on $R$ can be regarded as a $(\bC^*)^2$ action. Then $A_\alpha(U)$ is the ring of invariants under the second $\bC^*$ and hence is a finitely generated domain; and $M_{\alpha, \lambda}(U)$ is the degree $1$ piece under the second $\bC^*$ action (this can be interpreted as invariants of a twist by the $-1$ character) and hence is a finitely generated torsion-free module over $A_\alpha(U)$. Here we use Theorem~\ref{thm:GIT} in both cases.
\end{proof}

\begin{remark}
In fact, the above proposition holds without the restriction on $\dim(U)$.
\end{remark}

\begin{proof}[Proof of Theorem~\ref{mainthm}]
Let $U$, $V$, and $W$ be sufficiently large vector spaces satisfying $\dim(U) = \dim(V) \dim(W)$, and let $\omega \colon U \times V \times W \to \bC$ be a non-degenerate trilinear form. Then
\begin{align*}
(A_{\alpha}(U) \boxtimes A_{\beta}(V) \boxtimes A_{\gamma}(W))^{G(\omega)}
&= \bigoplus_{d \ge 0} (\bS_{d\alpha}(U) \otimes \bS_{d\beta}(V) \otimes \bS_{d\gamma}(W))^{G(\omega)} \\
&=\bigoplus_{d \ge 0} \bK_{d\alpha, d\beta, d\gamma}=B_{\alpha,\beta,\gamma}
\end{align*}
by Proposition~\ref{schur-kron}. Since $A_{\alpha}(U)$, $A_{\beta}(V)$, and $A_{\gamma}(W)$ are finitely generated graded integral domains (Proposition~\ref{prop:construct}), so is their Segre product (Corollary~\ref{cor:segre}). Since $G(\omega) \cong \GL(V) \times \GL(W)$ is a reductive group, the above invariant ring is also finitely generated (Theorem~\ref{thm:GIT}). This shows that $B_{\alpha,\beta,\gamma}$ is a finitely generated graded integral domain.

Similarly, we have
\begin{align*}
(M_{\alpha,\lambda}(U) \boxtimes M_{\beta,\mu}(V) \boxtimes M_{\gamma,\nu}(W))^{G(\omega)}
&= \bigoplus_{d \ge 0} (\bS_{d\alpha+\lambda}(U) \otimes \bS_{d\beta+\mu}(V) \otimes \bS_{d\gamma+\nu}(W))^{G(\omega)} \\
&=\bigoplus_{d \ge 0} \bK_{d\alpha+\lambda, d\beta+\mu, d\gamma+\nu}=N_{\alpha,\beta,\gamma}^{\lambda,\mu,\nu}.
\end{align*}
Since $M_{\alpha,\lambda}(U)$, $M_{\beta,\mu}(V)$, and $M_{\gamma,\nu}(W)$ are finitely generated torsion-free graded modules over $A_{\alpha}(U)$, $A_{\beta}(V)$, and $A_{\gamma}(W)$ (Proposition~\ref{prop:construct}), their Segre product is a finitely generated torsion-free graded module over the Segre product of the $A$'s (Corollary~\ref{cor:segre}). The invariant module is finitely generated over the invariant ring (Theorem~\ref{thm:GIT}), and obviously torsion-free. So $N_{\alpha,\beta,\gamma}^{\lambda,\mu,\nu}$ is a finitely generated torsion-free graded $B_{\alpha,\beta,\gamma}$-module.
\end{proof}

\begin{remark}
We could avoid discussing the Segre product by directly constructing the ring $B_{\alpha, \beta, \gamma}$ and module $N_{\alpha, \beta, \gamma}^{\lambda, \mu, \nu}$ using a triple product of flag varieties.
\end{remark}

\section{Remarks and complements} \label{sec:rmk}

\subsection{Algebraic properties} \label{sec:alg}

See \cite[\S 3]{kempf} for the definition of rational singularities.

\begin{proposition}
$B_{\alpha, \beta, \gamma}$ has rational singularities, so in particular is normal and  Cohen--Macaulay.
\end{proposition}

\begin{proof}
The ring $A_\alpha(U) \boxtimes A_\beta(V) \boxtimes A_\gamma(W)$ is the homogeneous coordinate ring of a homogeneous space for $\GL(U) \times \GL(V) \times \GL(W)$ and hence has rational singularities \cite[\S 2]{kempf}. This property is preserved by taking invariants under a reductive group \cite[Corollaire]{boutot}.
\end{proof}

We now give a proof of the ``only if'' direction of Conjecture~\ref{mainconj}. Assume that $(\alpha, \beta, \gamma)$ is stable. In particular, $g_{d\alpha, d\beta, d\gamma}$ is constant for $d \gg 0$. Furthermore, $g_{\alpha, \beta, \gamma} > 0$ by definition. Thus $B_{\alpha,\beta,\gamma}$ is a finitely generated graded normal domain whose Hilbert polynomial has degree~0 and whose first graded piece is non-zero. It follows that $B_{\alpha,\beta,\gamma} \cong \bC[t]$, with $t$ of degree one, and so $g_{d\alpha, d\beta, d\gamma} = 1$ for all $d > 0$.

\begin{remark}
In the above argument, we used that $g_{\alpha, \beta, \gamma} > 0$ is part of the definition for $(\alpha, \beta, \gamma)$ to be stable. If we dropped that assumption, we could have a case where $g_{d\alpha, d\beta, d\gamma} = 0$ for $d \gg 0$, and hence for all $d > 0$ since $B_{\alpha, \beta, \gamma}$ is a domain. This means that $B \cong \bC$ and hence that $g_{\lambda + d\alpha, \mu + d\beta, \nu + d\gamma} = 0$ for $d \gg 0$ for all $(\lambda, \mu, \nu)$. We cannot improve this to $d>0$: $g_{d\alpha, d\beta, d\gamma} = 0$ for all $d > 0$ if $\alpha = \beta = (3)$ and $\gamma = (1,1,1)$, but $g_{\lambda + \alpha, \mu + \beta, \nu + \gamma} = 1$ for $\lambda = \mu = \nu = (1,1,1)$.
\end{remark}

\begin{remark}
From what we have shown, $(\alpha, \beta, \gamma)$ is stable if and only if the Krull dimension of $B_{\alpha, \beta, \gamma}$ is $1$. Since $B_{\alpha, \beta, \gamma}$ is a ring of invariants, this property can be determined algorithmically. Although it is probably impractical, there did not seem to be an algorithm previously for determining stability.
\end{remark}

\begin{remark}
Define $G^{\alpha,\beta,\gamma}_{\lambda,\mu,\nu}$ to be the dimension of the vector space $N_{\alpha,\beta,\gamma}^{\lambda,\mu,\nu} \otimes \Frac(B_{\alpha, \beta, \gamma})$. (Note the swap of superscripts and subscripts.) If $(\alpha, \beta, \gamma)$ is stable then $G^{\alpha,\beta,\gamma}_{\lambda,\mu,\nu}$ is equal to the limiting value of $g_{\alpha+d\lambda,\beta+d\mu,\gamma+d\nu}$. In particular, $G_{\lambda,\mu,\nu}^{(1),(1),(1)}$ is the usual stable Kronecker coefficient $G_{\lambda,\mu,\nu}$. It would be interesting if one could say anything about the numbers $G^{\alpha,\beta,\gamma}_{\lambda,\mu,\nu}$ in general. For example, there is a well-known relationship between stable Kronecker coefficients and Littlewood--Richardson coefficients. Does this generalize in any way to the $G^{\alpha,\beta,\gamma}_{\lambda,\mu,\nu}$? 
\end{remark}

\begin{remark}
Since the function $d \mapsto g_{d\alpha, d\beta, d\gamma}$ is given by a quasi-polynomial $p_{\alpha, \beta, \gamma}(d)$ for all $d \ge 0$ (see \S\ref{ss:related-work}), one can ask if the rings $B_{\alpha,\beta,\gamma}$ have toric degenerations. This would imply that there is a rational polytope $Q(\alpha, \beta, \gamma)$ such that the $p_{\alpha, \beta, \gamma}(d)$ is the number of integer points in the $d$th dilate of $Q$ for all $d \ge 0$. Ehrhart reciprocity \cite[Theorem 4.6.9]{stanley} implies that $|p_{\alpha, \beta, \gamma}(-d)|$ is the number of integer points in the interior of the $d$th dilate of $Q$. So in particular, $p_{\alpha, \beta, \gamma}(d) \ge |p_{\alpha, \beta, \gamma}(-d)|$. 

However, from \cite[Theorem 2.4]{BOR}, this fails for $\alpha = (6,6)$, $\beta = (7,5)$, $\gamma = (6,4,2)$ where 
\[
p_{\alpha, \beta, \gamma}(d) = \begin{cases} (d+2)/2 & \text{if $d$ is even} \\ (d-1)/2 & \text{if $d$ is odd}\end{cases}.
\]
In particular, $p(1) = 0$ and $p(-1) = -1$. So no such polytopes exist in general. We thank Mateusz Micha{\l}ek for bringing this example to our attention.
\end{remark}

\subsection{Littlewood--Richardson coefficients}

If we replace $V \otimes W$ in \eqref{eq1} with $V \oplus W$, then we get a version of Theorem~\ref{mainthm} with Kronecker coefficients replaced by Littlewood--Richardson coefficients $c^\lambda_{\mu, \nu}$ (see \cite[(3.11)]{expos}). In this case, a different construction for the analog of the ring $B$ is given in \cite{LRpoly} and the rational singularities property was used to deduce that the function $d \mapsto c^{d\lambda}_{d\mu,d\nu}$ is a polynomial for $d \ge 0$ \cite[Corollary 3]{LRpoly}. This gives the following concrete statement (which can be deduced from previous work, see Remark~\ref{rmk:LR}):

\begin{theorem} \label{thm:LR}
Pick partitions $\alpha, \beta, \gamma$ with $|\beta|+|\gamma|=|\alpha|$. The following are equivalent:
\begin{enumerate}[\rm \indent (a)]
\item $c^\alpha_{\beta, \gamma} = 1$.
\item $c^{d\alpha}_{d\beta, d\gamma} = 1$ for all $d>0$.
\item For all partitions $\lambda, \mu, \nu$ with $|\mu| + |\nu| = |\lambda|$, $c^{\lambda + d\alpha}_{\mu + d\beta, \nu + d\gamma}$ is constant for $d \gg 0$.
\end{enumerate}
\end{theorem}

\begin{proof}
(a) and (b) are equivalent by \cite[\S 6.1]{KTW}. (c) implies (a) and (b) by taking $\lambda=\mu=\nu=\emptyset$ and using that $d \mapsto c^{d\alpha}_{d\beta,d\gamma}$ is a polynomial for all $d \ge 0$ and that $c^\emptyset_{\emptyset,\emptyset}=1$.

(b) implies (c) by following the same arguments in this paper.
\end{proof}

\begin{remark} \label{rmk:LR}
Regarding Theorem~\ref{thm:LR}:
\begin{enumerate}
\item The analogs of (a) and (b) for Kronecker coefficients are not equivalent: $g_{(2,1), (2,1), (2,1)} = 1$ while $g_{(4,2),(4,2),(4,2)} = 2$.
\item The equivalence of (b) and (c) can also be deduced from the fact that Littlewood--Richardson coefficients count the number of integer points in rational polytopes \cite{BZ} and using \cite[Proposition 4.4]{stembridge}. \qedhere
\end{enumerate}
\end{remark}

\subsection{Plethysm}

The technique used in this paper can also be used to prove stability properties of plethysm coefficients. The ideas are similar, so we omit the details. Let $a_{\lambda, \mu}^\nu$ be the multiplicity of $\bS_\nu$ in $\bS_\lambda \circ \bS_\mu$ (the composition of Schur functors). We emphasize that this is not symmetric in $\lambda$ and $\mu$.

\begin{theorem}
Fix partitions $\alpha, \beta, \gamma$. Assume that $a_{d\alpha, \beta}^{d\gamma} = 1$ for $d \ge 0$. Then for all $\lambda, \nu$, the function $d \mapsto a_{\lambda + d\alpha, \beta}^{\nu + d\gamma}$ is constant for $d \gg 0$.
\end{theorem}

\begin{proof}
Let $V$ be a vector space of large dimension. We can build a finitely generated integral domain
\[
C_{\alpha, \beta, \gamma} = (A_\alpha(\bS_\beta(V)) \boxtimes A_\gamma(V^*))^{\GL(V)} =  \bigoplus_{d \ge 0} (\bS_{d\alpha}(\bS_\beta(V)) \otimes \bS_{d\gamma}(V^*))^{\GL(V)}.
\]
and a finitely generated torsion-free $C_{\alpha, \beta, \gamma}$-module
\[
P_{\alpha, \beta, \gamma}^{\lambda, \nu} = (M_{\alpha, \lambda}(\bS_\beta(V)) \boxtimes M_{\gamma, \nu}(V^*))^{\GL(V)} = \bigoplus_{d \ge 0} (\bS_{d\alpha + \lambda}(\bS_\beta(V)) \otimes \bS_{d\gamma + \nu}(V^*))^{\GL(V)}.
\]
Note that $\dim(C_{\alpha, \beta, \gamma})_d = a_{d\alpha, \beta}^{d\gamma}$ and $\dim(P_{\alpha, \beta, \gamma}^{\lambda, \nu})_d = a_{\lambda + d\alpha, \beta}^{\nu + d\gamma}$.
\end{proof}

\begin{example}
\begin{enumerate}
\item Take $\gamma = \beta$ and $\alpha = (1)$. This is a result of Brion \cite{brion}.

\item Take $\beta = \alpha = (2)$ and $\gamma = (2,2)$. To see that $a^{d\gamma}_{d\alpha, \beta} = 1$ for all $d \ge 0$, note that we can take $V = \bC^2$. Then the ring of $\SL(2)$-invariants on $\Sym(\Sym^2(\bC^2))$ is a polynomial ring on one generator of degree $2$ (the discriminant). \qedhere
\end{enumerate}
\end{example}

\subsection{Twisted commutative algebras}

Murnaghan's stability theorem was reinterpreted in \cite[\S 3.4]{fimodules} as the fact that the Segre product of finitely generated FI-modules is finitely generated. This is a useful reformulation since it turns Murnaghan's numerical result into a structural result.

The rings $A_{\alpha}$ are examples of twisted commutative algebras (see \cite{expos} for an introduction to these objects), and modules over the tca $A_{(1)}$ are equivalent to FI-modules (see \cite[\S 1.3]{symc1}). One might therefore hope that the stability results in this paper could be reformulated as structural results for $A_{\alpha}$-modules. We believe that there is such a reformulation, though it is more complicated than the case of Murnaghan's theorem. Nonetheless, this point of view led to the proof in this paper. We plan to pursue the connection to tca's in future work.


\begin{thebibliography}{KTW}

\bibitem[BV]{BV} Velleda Baldoni, Michele Vergne, Multiplicity of compact group representations and applications to Kronecker coefficients, \arxiv{1506.02472v1}.

\bibitem[BZ]{BZ} A.~D.~Berenstein, A.~V.~Zelevinsky, Triple multiplicities for $sl(r+1)$ and the spectrum of the exterior algebra of the adjoint representation, {\it J. Algebraic Combin.} {\bf 1} (1992), no.~1, 7--22. 

\bibitem[Bo]{boutot} Jean-Fran\c{c}ois Boutot, Singularit\'es rationnelles et quotients par les groupes r\'eductifs, {\it Invent. Math.} {\bf 88} (1987), no.~1, 65--68. 

\bibitem[BOR]{BOR} Emmanuel Briand, Rosa Orellana, Mercedes Rosas, Reduced Kronecker coefficients and counter-examples to Mulmuley's strong saturation conjecture SH, with an appendix by Ketan Mulmuley, {\it Comput. Complexity} {\bf 18} (2009), no.~4, 577--600, \arxiv{0810.3163v3}.

\bibitem[Br]{brion} Michel Brion, Stable properties of plethysm: on two conjectures of Foulkes, {\it Manuscripta Math.} {\bf 80} (1993), no. 4, 347--371.

\bibitem[CEF]{fimodules} Thomas Church, Jordan Ellenberg, Benson Farb, FI-modules and stability for representations of symmetric groups, {\it Duke Math. J.} {\bf 164} (2015), no.~9, 1833--1910, \arxiv{1204.4533v4}.

\bibitem[DW]{LRpoly} Harm Derksen, Jerzy Weyman, On the Littlewood-Richardson polynomials, {\it J. Algebra} {\bf 255} (2002), no.~2, 247--257.

\bibitem[Fu]{fulton} William Fulton, {\it Young Tableaux}, London Mathematical Society Student Texts {\bf 35}. Cambridge University Press, Cambridge, 1997.

\bibitem[Ke]{kempf} George R. Kempf, On the collapsing of homogeneous bundles, {\it Invent. Math.} {\bf 37} (1976), no.~3, 229--239. 

\bibitem[KTW]{KTW} Allen Knutson, Terence Tao, Christopher Woodward, The honeycomb model of $GL_n(\mathbb{C})$ tensor products. II: Puzzles determine facets of the Littlewood-Richardson cone, {\it J. Amer. Math. Soc.} {\bf 17} (2004), no.~1, 19--48, \arxiv{math/0107011v2}.

\bibitem[L]{littlewood} D.~E.~Littlewood, Products and plethysms of characters with orthogonal, symplectic and symmetric groups, {\it Canad. J. Math.} {\bf 10} (1958), 17--32.

\bibitem[Ma1]{manivel1} Laurent Manivel, On the asymptotics of Kronecker coefficients, \arxiv{1411.3498v1}.

\bibitem[Ma2]{manivel2} Laurent Manivel, On the asymptotics of Kronecker coefficients 2, \arxiv{1412.1782v1}.

\bibitem[MS]{MS} Eckhard Meinrenken, Reyer Sjamaar, Singular reduction and quantization, {\it Topology} {\bf 38} (1999), no.~4, 699--762, \arxiv{dg-ga/9707023v1}.

\bibitem[Mu]{murnaghan} F.~D.~Murnaghan, The analysis of the Kronecker product of irreducible representations of the symmetric group, {\it Amer. J. Math.} {\bf 60} (1938), no.~3, 761--784.

\bibitem[PP]{pak-panova} Igor Pak, Greta Panova, Bounds on the Kronecker coefficients, \arxiv{1406.2988v2}.

\bibitem[PV]{PV} V.~L.~Popov, E.~B.~Vinberg, Invariant Theory, in: {\it Algebraic Geometry IV}, Encyclopaedia of Mathematical Sciences {\bf 55}, Springer-Verlag, 1994, 123--278.

\bibitem[SS1]{symc1} Steven~V Sam, Andrew Snowden, GL-equivariant modules over polynomial rings in infinitely many variables, {\it Trans. Amer. Math. Soc.}, to appear, \arxiv{1206.2233v2}.

\bibitem[SS2]{expos} Steven~V Sam, Andrew Snowden, Introduction to twisted commutative algebras, \arxiv{1209.5122v1}.

\bibitem[Sta]{stanley} Richard~P. Stanley, {\it Enumerative Combinatorics. Volume 1}, second edition, Cambridge Studies in Advanced Mathematics {\bf 49}, Cambridge University Press, Cambridge, 2012.

\bibitem[Ste]{stembridge} John R. Stembridge, Generalized stability of Kronecker coefficients, available from \url{http://www.math.lsa.umich.edu/~jrs/}.

\bibitem[V]{vallejo} Ernesto Vallejo, Stability of Kronecker coefficients via discrete tomography, \arxiv{1408.6219v1}.

\end{thebibliography}
\end{document}